\documentclass[12pt]{article}

\usepackage{amsmath,amsthm,amssymb}

\def\real{\mathbb R}
\def\complex{\mathbb C}

\def\Hilbert{{\mathbb H}}

\def\graph{{\mathcal G}}
\def\gbar{\overline{\mathcal G}}
\def\vertexset{{\mathcal V}_{\graph}}
\def\edgeset{{\mathcal E}_{\graph}}

\def\alg{{\mathcal A}}

\def\Laplace{{-D^2}}
\def\dop{{\mathcal L}}
\def\domain{{\mathcal D}}

\title{Robin boundary conditions \\
for the Laplacian \\
on metric graph completions}

\author{Robert Carlson
\\Department of Mathematics
\\University of Colorado at Colorado Springs
\\Colorado Springs, Colorado USA
\\rcarlson@uccs.edu}

\date{2021}

\begin{document}


\newtheorem{thm}{Theorem}[section]
\newtheorem{cor}[thm]{Corollary}
\newtheorem{lem}[thm]{Lemma}
\newtheorem{prop}[thm]{Proposition}
\newtheorem{ax}{Axiom}

\theoremstyle{definition}
\newtheorem{defn}{Definition}[section]

\theoremstyle{remark}
\newtheorem{rem}{Remark}[section]
\newtheorem*{notation}{Notation}


\newcommand{\thmref}[1]{Theorem~\ref{#1}}
\newcommand{\secref}[1]{\S\ref{#1}}
\newcommand{\lemref}[1]{Lemma~\ref{#1}}
\newcommand{\propref}[1]{Proposition~\ref{#1}}
\newcommand{\corref}[1]{Corollary~\ref{#1}}
\newcommand{\figref}[1]{Figure~\ref{#1}}

\numberwithin{equation}{section}


\maketitle

\begin{abstract}

A generalization of Robin boundary conditions leading to self-adjoint operators is developed for the second derivative operator
on metric graphs with compact completion and totally disconnected boundary.  Harmonic functions and 
their properties play an essential role.

\end{abstract}

Keywords:  quantum graph, harmonic functions on graphs, boundary value problems

MSC-class: 34B45 (primary), 47E05 (secondary)

\newpage

\section{Introduction}

Let $\graph $ denote a connected, locally finite metric graph with a countable vertex set $\vertexset$ and edge set $\edgeset$. 
Edges $e \in \edgeset $ are assigned a length $l_e$ and are identified with intervals $[a_e,b_e]$ of length $l_e$.  
With the usual geodesic distance, $\graph $ becomes a metric space.

The boundary $\partial \graph$ of $\graph $ will be a subset of vertices, including all vertices with degree $1$.
The interior $\graph _{int}$ of $\graph $ will be the complement of the boundary vertices.
As a metric space, ${\graph}$ has a completion $\gbar $; the boundary of $\gbar $ will be the complement of the graph interior, 
$\partial \gbar  = \gbar  \setminus \graph _{int}$.  

The boundary and function theory developed here are based on two essential assumptions: $\gbar$ is compact and $\partial \gbar $ is totally disconnected.
Simple examples, with boundary homeomorphic to the Cantor set, may be constructed from homogeneous trees with decaying edge lengths.  
A sufficient (but not necessary) condition for these properties is that the volume, which is the sum of the edge lengths, is finite.  
In the finite volume cases $\gbar $ is \cite{Georg11} the end compactification of $\graph $.

When $\graph $ is finite, various authors have characterized boundary conditions leading to a self-adjoint Laplace differential operator $\Laplace$ on $L^2(\graph )$.  
When $\graph $ is infinite and the edge lengths of have a positive lower bound, the existence of a unique self-adjoint extension of a 'minimal' symmetric operator $\Laplace $ is common.  A useful discussion and numerous references are in \cite{BK}; a more recent source with additional information is \cite{Exner18}.
 
For infinite graphs having compact completions $\gbar $ with totally disconnected boundary $\partial \gbar $, 'minimal' symmetric operators $\Laplace $ on $L^2(\graph )$
satisfying standard Kerckhoff conditions at interior vertices may have many distinct self-adjoint extensions.  
Recent works addressing related questions about self-adjoint operators include \cite{GHKLW} and \cite{KMN}.
Physical models motivate a search for such extensions characterized by 'boundary conditions'.  
This search leads to novel problems, especially when the boundary conditions describe behavior at points in $\partial \gbar $
that are not vertices of  $\graph $.
Some of these boundary conditions and corresponding operators were described in \cite{Carlson08}.  Initial domains there consisted of functions which either 
(i) vanished outside compact subsets of $\graph _{int}$, 
or (ii) had derivatives vanishing outside compact subsets of $\graph _{int}$.  Such domains extend the classical Dirichlet or Neumann boundary conditions.
The symmetric operators $\Laplace $ with these domain are nonnegative, so have self-adjoint Friedrichs extensions.

The main goal of this work is to identify and study a suitable generalization of 'mixed' or Robin boundary conditions leading to self-adjoint Laplace operators.
Consider the classical second derivative operator $-D^2$ acting on  $L_{\real }^2[0,1]$ with the boundary conditions $f'(0) = \alpha f(0)$ and $f'(1) = \beta f(1)$.
To satisfy these boundary conditions, start with a domain consisting of smooth functions which have the form $ c_0(\alpha x + 1)$ in some neighborhood of $x=0$, and   
$c_1(\beta x + 1 - \beta )$ near $x=1$.  This domain, which is a core for a self-adjoint operator, is defined with the aid of functions which are harmonic near the boundary. 
This simple example will be generalized to build domains for symmetric and self-adjoint operators $\Laplace$ on $L^2(\graph )$.
Results describing the existence and properties of harmonic functions on $\gbar $ play an essential role. 

The results are developed in three subsequent sections.  Section 2 begins with a review of basic material on metric graphs.
Some results about compact totally disconnected metric spaces such as $\partial \gbar $ are then presented, along with a theorem which
links the totally disconnected boundary with a 'weakly connected' condition for $\gbar $ which appeared in \cite{Carlson08}.
Section 3 treats the existence and properties of harmonic functions on $\gbar$.  The introduction of energy spaces 
provides a new approach to solving the Dirichlet problem for metric graphs.  Level sets of harmonic functions are considered;
these help provide needed refinements of the existence results.  Section 4 then addresses the construction of symmetric and
nonnegative self-adjoint Laplace operators based on novel boundary conditions, defined with the aid of harmonic functions.  
The quadratic forms for these operators include boundary terms
which distinguish them from the Dirichlet and Neumann cases.
 
\section{Graphs with totally disconnected boundary}

\subsection{Metric graphs}

Suppose $w_1, w_2 \in \vertexset$.  A vertex path from $w_1$ to $w_2$ is a finite vertex sequence $v_1,\dots ,v_N$
with $v_1 = w_1$, $v_N = w_2$, and $v_n$ adjacent to $v_{n+1}$ for $n  = 1,\dots ,N-1$.  If the edge $e_n$ joining $v_n$ to $v_{n+1}$ 
has length $l_n$ and $L =  \sum_{n=1}^{N-1} l_n$, then a path $\gamma $ from $w_1$ to $w_2$ (with length $L$) is the function $\gamma : [0,L] \to \graph $ obtained by
traversing the edges $e_n$ from $v_n$ to $v_{n+1}$ and $n=1,\dots ,N-1$.  

If $e = \{ v_1,v_2 \} \in \edgeset$ is identified with the interval $[a,b]$ and $x \in e$ is not a vertex, it may be useful to
treat $x$ as an added vertex adjacent to $v_1,v_2$.  Then identify $\{v_1,x\}$ with $[a,x]$ and $\{ x,v_2 \}$ with $[x,b]$.
A path joining two such points $x_1$ and $x_2$ may be defined as above. 
The distance $d(x_1,x_2)$ between points $x_1$ and $x_2$ in $\graph $ is defined as the infimum of the lengths of paths joining $x_1$ and $x_2$.  
This metric extends continuously to $\gbar $.

Points $x_1,x_2 \in \gbar$ can be joined by possibly infinite paths $\gamma :[0,L) \to \graph $ or $\gamma :(-L,L) \to \graph $ 
of finite length, defined analogously for sequences $v_1,v_2, v_3, \dots $ or bidirectional sequences $\dots ,v_{-2},v_{-1},v_0,v_1, \dots$.
In this case it is assumed that $x_2 = \lim_{t \to L} \gamma (t)$ and $x_1 = \gamma (0)$ or $x_1 = \lim_{t \to -L} \gamma (t)$ as appropriate. 

Metric graphs may be equipped with a variety of function spaces.
A function $f:\graph \to \real $ has components $f_e:[a_e,b_e] \to \real$.
In this work functions are real-valued unless otherwise noted.
Our basic Hilbert space is the usual Lebesgue space
\[L^2(\graph ) = \oplus _{e \in \edgeset} L^2[a_e,b_e], \]
with inner product 
\[\langle f,g \rangle _2 =  \int_{\graph} fg = \sum_{e \in \edgeset} \int_{a_e}^{b_e} f_eg_e .\] 

The notation $e \sim v$ indicates that the edge $e$ is incident on a vertex $v$.
If $e \sim v$, the notation $\partial _{\nu} f_e(v)$ is used to indicate the derivative of $f_e$ at $v$ computed in
outward pointing local coordinates.  That is, for this computation, the identification of $e$ with $[a_e,b_e]$,
identifies $v$ with $a_e$.     
  
As an initial domain for $\Laplace $, let $\domain _{max}$ denote the continuous real valued functions $f$ on $\graph $ 
which have absolutely continuous derivatives on each edge $e$, with $f$ and $f'' \in L^2(\graph )$, and which satisfy 
\begin{equation} \label{Kerck} 
\sum _{e \sim v} \partial _{\nu} f_e(v) = 0
\end{equation} 
at interior vertices $v \in \vertexset$.

The interior vertex condition \eqref{Kerck} leads to an important integration by parts lemma.

\begin{lem} \label{IBP}
Suppose $\graph $ is a finite graph with boundary $\partial \graph $.  
If $f, g \in \domain _{max}$, then
\begin{equation} \label{ibp}
\int_{\graph} f''g = - \sum_{v \in \partial \graph} \sum_{e \sim v} g(v) \partial _{\nu} f_e(v) - \int_{\graph} f'g' 
\end{equation}
\[= \sum_{v \in \partial \graph} \sum_{e \sim v} [ f(v) \partial _{\nu} g_e(v) - g(v) \partial _{\nu} f_e(v)] +   \int_{\graph} f g''\]
\end{lem}

\begin{proof}
Using the identification of edges $e$ with intervals $[a_e,b_e]$, integration by parts gives
\[ \int_{\graph} f''g = \sum_{e \in \edgeset} \int_{a_e}^{b_e} f_e''g_e = \sum_{e \in \edgeset}  \Big [ f_e'(b_e)g_e(b_e) - f_e'(a_e)g_e(a_e) \Big ]   - \int_{\graph } f'g' \]
Regroup the sum of boundary terms, collecting those having evaluation at the same vertex.
The terms  $ f_e'(b_e)g_e(b_e)$ have derivatives computed in inward pointing local coordinates, and $g$ is continuous at each $v$, so  
\[ \int_{\graph} f''g = - \sum_{v \in \vertexset} g(v) \sum_{e \sim v}  \partial _{\nu} f_e(v) - \int_{\graph } f'g' \]

\eqref{Kerck} implies that terms in the last sum coming from interior vertices vanish, giving the first line of \eqref{ibp}.
Another integration by parts produces the second line.

\end{proof}

\subsection{Totally disconnected boundary}

Recall that $\gbar$ is compact and $\partial \gbar $ is totally disconnected. 
Since $\gbar$ is compact it must be totally bounded,  leading to the following observation.

\begin{prop}
$\gbar$ is compact if and only if for every $\epsilon > 0$ 
there is a finite subgraph $\graph _{0}$ of $\graph $, 
such that for every $y \in \graph $ there is an $x \in \graph _{0}$ with $d(x,y) < \epsilon $.
\end{prop}

Given $\epsilon > 0$, it will be convenient to have a subgraph $\graph _{\epsilon}$ of $\graph$ containing all points $x \in \graph $ whose 
distance from $\partial \gbar $ is at least $\epsilon $.  The edges $e$ of $\graph _{\epsilon }$ are the (closed) edges of $\graph $ containing a point $x$ with
$d(x,\partial \gbar) \ge \epsilon $.  If an edge $e$ of $\graph _{\epsilon }$ is incident in $\graph $ on a vertex $v$, then $v$ is a vertex of $\graph _{\epsilon}$.

\begin{lem} \label{approx}
If $\gbar $ is compact then $\graph _{\epsilon}$ is a finite graph.
\end{lem}

\begin{proof}
Arguing by contradiction, suppose $\graph $ has
a sequence $\{ e_n, n = 1,2,3,\dots  \} $ of distinct edges, each containing a point $x_n$ with $d(x_n,\partial \gbar ) \ge \epsilon $.
Since $\gbar $ is compact, the sequence $\{ x_n \}$ has a convergent subsequence $\{ z_k , k = 1,2,3,\dots \}$ with limit $z \in \gbar $.
Since $d(z, \partial \gbar ) \ge \epsilon $ it must be that $z \in \graph $.  

Suppose $z$ is in the edge $e$.  Since $\graph $ is locally finite, there are only finitely many edges
sharing a vertex with $e$.  With only finitely many exceptions, the points $z_k$ are outside this set of edges, so
$z$ cannot be the limit of the subsequence.  Thus $\graph _{\epsilon }$ cannot have infinitely many distinct edges.

\end{proof}

The boundary $\partial \gbar$ is a closed subset of $\gbar$, so it a totally disconnected compact metric space.
Some general facts about a such metric spaces will be useful.  In particular,
as a totally disconnected compact metric space, $\partial \gbar $ will have a rich collection of clopen subsets, 
which are both open and closed in $\partial \gbar $.  

A version of the next result about a totally disconnected compact metric space $\Omega$ is in \cite[p. 97]{Hocking}.
Suppose ${\cal E }_1$ and ${\cal E }_2$ are partitions of $\Omega$.
Partition ${\cal E }_2$ is a refinement of $ {\cal E }_1$ if each set in ${\cal E }_2$ is a subset of a set in $ {\cal E }_1$.

\begin{prop} \label{manycor}
Suppose $\Omega $ is a totally disconnected compact metric space.  
For any $\epsilon > 0$, there is a finite partition ${\cal E} = \{ E(n), n=1, \dots ,N \}$ of $\Omega $ by clopen sets such that 
the diameter of each $E(n) \in {\cal E}$ is less than $\epsilon $.

There is a sequence $\{ {\cal E }_j, j = 1,2,3,\dots \} $ of partitions of $\Omega $ by clopen sets, with ${\cal E }_{j+1}$ a refinement of ${\cal E }_j$,
such that the diameter of each set in $ {\cal E }_j $ is less than $1/j$. 
\end{prop}

The next result characterizes the compact completions $\gbar $ with totally disconnected boundary.
In an earlier work \cite{Carlson08} the author used an assumption that $\gbar $ was weakly connected.
$\gbar $ is weakly connected if for every pair of distinct points $x,y \in \gbar $,
there is a finite set of points $W = \{ w_1,\dots ,w_K \} $ in the graph $\graph $ 
separating $x$ from $y$.  That is, there are disjoint open subsets $U_x,U_y$ of $\gbar$ with $x \in U_x$ and $y \in U_y$ such that
$\gbar \setminus W = U_x \cup U_y$.

\begin{thm} \label{tdeqwc}
If $ \gbar $ is compact,  then $\partial \gbar $ is totally disconnected if and only if 
$\gbar $ is weakly connected.
\end{thm}

\begin{proof}
Assume $\gbar $ is weakly connected. If $x$ and $y$ are distinct points in $\partial \gbar $,
they are in distinct clopen subsets of $\partial \gbar$ given by $U_x \cap \partial \gbar $ and $U_y \cap \partial \gbar $, with 
$[U_x \cap \partial \gbar ] \cup  [U_y \cap \partial \gbar ] = \partial \gbar$. Thus $x$ and $y$ lie in distinct connected components,
and $\partial \gbar$ is totally disconnected.

Assume now that $\gbar $ is compact with a totally disconnected boundary.  
Consider distinct points $x,y$ in $\gbar $. 
If either $x$ or $y$ is a point of $\graph $, they are easily separated by the removal of a finite set of points in $\graph $.
Assume then that $x$ and $y$ belong to $\partial \gbar $, but are not boundary vertices of $\graph $.
Then by \propref{manycor} there are disjoint clopen sets $E_x,E_y \subset \partial \gbar$ 
with $x \in E_x$, $y \in E_y$, and $E_x \cup E_y = \partial \gbar$.

For $z  \in \gbar $, let $B_{\epsilon}(z)$ denote the open ball of radius $\epsilon > 0$ centered at $z$, 
while the $\epsilon $ neighborhood of a set $E \subset \gbar$ is $N_{\epsilon}(E) = \cup_{z \in E} B_{\epsilon}(z) $.
Since $E_x$ and $E_y$ are compact and disjoint in $\gbar $, the neighborhoods $N_{\epsilon}(E_x)$ and $N_{\epsilon}(E_y)$
are disjoint if $\epsilon > 0$ is sufficiently small \cite[p. 86]{KolmF}.

The subgraph $\graph _{\epsilon} $, which is finite by \lemref{approx}, is now useful.
The set $\gbar \setminus [ N_{\epsilon}(E_x) \cup N_{\epsilon}(E_y)] $ is a subset of $\graph _{\epsilon}$.
Define 
\[U_x =  N_{\epsilon}(E_x) \setminus \graph _{\epsilon}, \quad U_y =  N_{\epsilon}(E_y) \setminus \graph _{\epsilon}.\]
Note that $U_x \cup U_y \cup \graph _{\epsilon} = \gbar $. 
The sets $U_x$ and $U_y$ are still open neighborhoods of $x,y$ respectively. 

Let $W$ be the set of vertices in $\graph _{\epsilon}$, and let $V$ be the complement of $U_x$ in $\gbar \setminus W$.
$V$ is open since it is the union of $U_y$ and the collection of open edges of $\graph _{\epsilon}$. 
The sets $U_x,V$ provide the desired separation of $x$ and $y$ by a finite set $W$ of points from $\graph $, showing that $\gbar $ is weakly connected.
\end{proof} 

An important role in the function theory of $\gbar$ is played by an algebra $\alg$ of 'eventually flat' functions.
$\alg $ is the set of functions $\phi: \gbar \to \real $ which are continuous on $\gbar$ and infinitely differentiable on the open edges of $\graph $,  
with $\phi ' = 0$ in the complement of a finite collection of edges, and 
in an open neighborhood of each vertex $v \in \graph$.  
With pointwise multiplication, $\alg$ is a subalgebra of the continuous functions
on $\gbar$ which contains the constant functions. 
A similar class of functions and its relation to the end compactification of a graph was
considered in \cite{Cartwright}.  

\begin{lem} \label{Afunk}
Suppose $\gbar$ is compact with a totally disconnected boundary.
Assume that $E$ and $E^c = \partial \gbar \setminus E$ are nonempty clopen subsets of $\partial \gbar$.
Then there is a function $\phi \in \alg$ with $\phi (x) = 1$ for $x \in E$ and $\phi (x) = 0$ for $x \in E^c$. 
\end{lem}

\begin{proof}
Since $E$ and $E^c$ are disjoint and compact, we may choose $\epsilon > 0$ such that $d(x,y) > 3 \epsilon $ for all
$x \in E$ and $y \in E^c$.  Begin by taking $\phi (x) = 1$ if $d(x,E) \le \epsilon $ and $\phi (x) = 0$ if $d(x,E^c) \le \epsilon $.
Now $\phi $ must be extended to $\graph _{\epsilon }$.

Let $e$ be a closed edge of $\graph _{\epsilon }$ which contains a point $x$ where $\phi $ is not yet defined.
First, if $e$ has no point $x$ with $\phi (x) = 1$ define $\phi (x) = 0$ for $x \in e$.  For the remaining edges,
$\phi ^{-1}(1) \cap e$ and $\phi ^{-1}(0) \cap e$ will be separated subintervals of $e$, each containing an endpoint of $e$.  Extend $\phi $ smoothly to these edges $e$,
with $\phi ' = 0$ in a neighborhood of the endpoints of $e$.  Since $\graph _{\epsilon }$ is finite by \lemref{approx}, the extended function $\phi $
is in $\alg $.
 
\end{proof}

\propref{manycor} easily shows that $\alg$ separates points of $\gbar$, so the Stone-Weierstrass theorem implies the next result.

\begin{thm} \label{A1}
Assume $\gbar$ is compact with totally disconnected boundary $\partial \gbar$.
Then $\alg $ is uniformly dense in the space of continuous functions on $\gbar $, and the boundary values of 
$\alg $ are uniformly dense in the space of continuous functions on $\partial \gbar $.
\end{thm}

\section{Finite energy harmonic functions}

\subsection{Energy spaces}

The introduction of an additional Hilbert space $\Hilbert _1$ will assist in understanding the harmonic functions on $\gbar$.
Let $\mu $ be a finite positive measure on $\partial \gbar $. The $\Hilbert _1$ inner product is 
\begin{equation} \label{Eprod}
\langle f,g \rangle _{1} =  \int_{\graph}  f'g' + \int_{\partial \gbar} fg \ d \mu .
\end{equation}
The elements of $\Hilbert _1$ are the functions $f: \gbar \to \real $ which are continuous on $\gbar$ and absolutely continuous on the edges of $\graph $,
with $f' \in L^2(\graph )$.  Addition and scalar multiplication are defined pointwise as usual.
Similar spaces appear in the study of resistor networks \cite{Doyle} and the associated operator theory \cite{JP}.
The measure $\mu $ plays a modest role in this work, but can have significance in physical modeling, as the next example illustrates.

Examples incorporating Robin boundary conditions and energy space inner products arise from models of strings coupled to springs.
Suppose the string displacement from equilibrium is $u(t,x)$, with $a \le x \le b$.
The string is attached to a spring at $a$ with spring constant $k_a > 0$, and at $b$ with constant $k_b > 0$.
The springs are  constrained to move transversely.   A standard model  \cite[p. 30]{Graff} for the system motion
uses the wave equation $ \partial ^2 u/\partial t^2 - \partial ^2 u/\partial x^2 = 0$ with the boundary conditions
\begin{equation} \label{Robin0}
 \frac{\partial u(t,x)}{\partial x}|_{x = a}  = k_a u(t,a), \quad  \frac{\partial u(t,x)}{\partial x}|_{x = b}  = -k_b u(t,b).
\end{equation}
The associated Sturm-Liouville operator $-D^2 = - \partial ^2 /\partial x^2$ with the boundary conditions $f'(a) = k_af(a)$ and
$f'(b) = -k_bf(b)$ has the quadratic (potential energy) form
\begin{equation} \label{Xform}
\int_a^b (-D^2f)f \ dx = -f'f \big |_a^b + \int_a^b (f')^2  \ dx 
\end{equation}
\[= k_bf(b)^2 + k_af(a)^2  + \int_a^b (f')^2  \ dx,\]
which is strictly positive for $f \not= 0$.
This is just the energy space form if the graph has a single edge and the measure $\mu$ assigns mass $k_a$ to $a$ and $k_b$ to $b$.

These models have a direct generalization  to finite graphs $\graph _0$, with continuity and $\eqref{Kerck}$ 
holding at interior vertices.  At boundary vertices the Robin condition becomes
\begin{equation} \label{Robin1} 
\sum _{e \sim v} \partial _{\nu} f_e(v) = k_v f(v),
\end{equation} 
with integration giving
\[\int_{\graph _0}  (-D^2f)f = \sum_{v \in \partial \graph _0} k_v f(v)^2 + \int _{\graph _0} (f')^2 .\]
Without the spring contributions, string vibration problems on finite graphs have been studied before \cite{Avdonin, DagerZ,Lagnese}. 

Returning to the general case, the next result establishes that $\Hilbert _1$ is complete.

\begin{prop} \label{isHilbert}
Assume that $\mu $ is a positive measure on $\partial \gbar$ with $0 < \mu ( \partial \gbar ) < \infty$.
A Cauchy sequence in $\Hilbert _1$ converges uniformly to a continuous function on $\gbar $.
$\Hilbert _1$ is a Hilbert space.
\end{prop}
 
\begin{proof}
Suppose first that $\| f \| _1^2 = \langle f,f \rangle _1 = 0$.  Then $f = 0$, $\mu$-almost everywhere  in $\partial \gbar$.
Pick $x_0 \in \partial \gbar $ with $f(x_0 ) = 0$.  For $x \in \gbar $, pick a path $\gamma $ of finite length joining $x_0$ to $x$.
Since $f$ is absolutely continuous and $f' = 0$ as an element of $L^2(\graph)$,
\[f(x) = \int_{\gamma } f'(t) \ dt = 0.\]
The form \eqref{Eprod} is thus positive definite, and so defines an inner product.  

Suppose $\{ f_n \}$ is a Cauchy sequence in $\Hilbert _1$, with
\[ \int_{\graph } (f'_n - f'_m)^2 +  \int_{\partial \gbar} (f_n-f_m)^2 \ d \mu \to 0, \quad m,n \to \infty .\]
Given $\epsilon > 0$ and $m,n$ large, there are points $x_0 \in \partial \gbar $ with $|f_n(x_0) - f_m(x_0)| < \epsilon $. 
For $x \in \gbar$, integration over a path $\gamma $ from $x_0$ to $x$ with length at most $2*{\rm diam}(\graph )$ gives
\[ |f_n(x) - f_m(x)| \le | f_n(x_0) - f_m(x_0)| + | \int_{x_0}^x f_n'(t) - f_m'(t) \ dt |\]
\[ \le \epsilon + 2{\rm diam}( \graph )^{1/2} \| f_n - f_m \|_1   .\]
Thus $\{ f_n \}$ is a uniformly convergent sequence of continuous functions on $\gbar $, with a continuous limit $f$.
 
Since $L^2(\graph )$ is complete, the sequence $\{f_n' \}$ converges in $L^2(\graph )$ to a function $g$.
Integration again gives
\[ f_n(x) - f_n(x_0) = \int_{\gamma } f_n'(t) \ dt .\]
The $L^2(\graph )$ convergence of $f_n'$ to $g$ implies $L^1$ convergence on the path from $x_0$ to $x$, so
\[f(x) - f(x_0) = \int_{x_0}^x g(t) \ dt .\]
Thus $f$ is absolutely continuous \cite[p. 110]{Royden}, $g(x) = f'(x)$ almost everywhere, and $\Hilbert _1$ is complete.

\end{proof}

\subsection{Harmonic functions}

A continuous function $f:\gbar \to \real $ is harmonic if (i) $D^2 f = 0$ on each edge, so $f$ is piecewise linear,
and (ii) $f$ satisfies the standard vertex conditions \eqref{Kerck} at each interior vertex.  
Say that a harmonic function $f$ has finite energy if $f \in \Hilbert _1$.  
Let $H_{fin}$ denote the set of finite energy harmonic functions on $\gbar $. 

Harmonic functions satisfy the mean value property at interior vertices.  
Assume that $v$ is an interior vertex with $N$ incident edges $e = [a_e,b_e]$ pointing away from $v$
so that each $a_e$ is identified with $v$.  
Suppose that $f$ is a function which is continuous at $v$, and whose restriction $f_e$ to $e$ is linear.
The identity  $f_e(a_e) = f_e(x) - (x - a_e)f_e' $ holds for $a_e \le x \le b_e$.
If $x \le \min _{e \sim v} (b_e)$ and $x-a_e$ has the same value on each edge, then 
\[ f(v) = \frac{1}{N} \sum_{n=1}^N f_e(x) - \frac{x-a_e}{N} \sum_{e \sim v} \partial _{\nu} f_e ,\]
from which the next result is obtained.     
     
\begin{lem}
If $f$ is linear on the edges incident on $v$ and continuous at $v$, 
then $f(v)$ is the mean value of the equidistant edge values $f_e(x)$ 
if and only if $\sum_{n=1}^N \partial _{\nu} f_e = 0$.
\end{lem}

Recall that $\gbar $ is path connected and compact.
The mean value property for a harmonic function $f$ on a metric graph means that $f$ has an interior maximum or minimum 
on $\gbar $ if and only if $f$ is constant.  Moreover, $f$ has a maximum and minimum, which must occur on the $\partial \gbar$.

\subsection{The Dirichlet problem}

Let $\domain _{min}$ denote the continuous functions $f: \gbar \to \real $, with compact support in the interior of $\graph$,
which are infinitely differentiable on the (closed) edges of $\graph$, and which   
satisfy \eqref{Kerck} at each interior vertex.  $\domain _{min}$ is dense in $L^2(\graph )$, but the situation is different in $\Hilbert _1$.

\begin{prop} \label{R1}
Functions in the $\Hilbert _1$ closure of $\domain _{min}$ vanish on $\partial \gbar$.
The orthogonal complement of $\domain _{min}$ in $\Hilbert _1$ is the set $H_{fin}$ of finite energy harmonic functions.
\end{prop}

\begin{proof}
By \propref{isHilbert}, convergence in $\Hilbert _1$ implies uniform convergence,  so functions in the closure of $\domain _{min}$ vanish on $\partial \gbar$.

First suppose that $f \in \domain _{min}$ and $g \in H_{fin}$.  Since $f$ vanishes outside a finite graph, the integration by parts formula \lemref{IBP} 
yields  $\langle g,f \rangle _1 = -\int_{\graph } fg'' = 0$. The finite energy harmonic functions are thus orthogonal to $\domain _{min}$ in $\Hilbert _1$.

Suppose $g \in \Hilbert _1$ and for all $f \in \domain _{min}$ 
\[ \langle g,f \rangle _1 = \int_{\graph } g'f' + \int_{\partial \gbar} gf \ d\mu = 0.\]
The boundary integral is zero, so will not play a role.

Each edge $e \in \edgeset$ is identified with an interval $[a,b]$.
Consider the functions $f \in \domain _{min}$ with support in $(a,b)$. 
For such $f $, integration by parts gives
\[0 = \int_a^b g'f' = - \int_a^b gf'' .\]
As a function in $L^2[a,b]$ the restriction of $g$ to $[a,b]$ is orthogonal to all such $f''$, which implies \cite[p. 1291]{Dunford} that 
$g'' = 0$ on each edge.

Recall that functions in $\Hilbert _1$ are continuous on $\gbar$ by definition.  Returning to more general $f \in \domain _{min}$, suppose 
that $f$ is a nonzero constant in a small neighborhood of an interior  vertex $v$, and the support of $f$ lies in the union of the edges incident on $v$.
Since  $\langle g,f \rangle _1 = 0$, summing over the edges incident on $v$ gives
\[0 = \sum_{e \sim v} \int_{e} f'g' = - f(v) \sum_e \partial _{\nu} g_e(v) -  \sum_{e \sim v} \int_{e} fg'' .\]
But $g'' = 0$ on each edge, so $g$ satisfies the vertex conditions \eqref{Kerck}.  That is, $g \in H_{fin}$. 

\end{proof} 

\begin{cor} \label{minzer}
Suppose $f \in \Hilbert _1$.  Among all $g \in \Hilbert _1$ which agree with $f$ on $\partial \gbar $,
a unique harmonic function minimizes $\| g \| _1$.

If $f \in \Hilbert _1$ vanishes on $\partial \gbar $, then $\langle f, h \rangle _1 = 0$ for all harmonic functions $h \in \Hilbert _1$. 
\end{cor}

\begin{proof}
Write $g = g_1+g_2$ with $g_1$ in the $\Hilbert _1$ closure of $\domain _{min}$ and $g_2 \in H_{fin}$.
As noted in the proof of \lemref{isHilbert}, a function in the $\Hilbert _1$ closure of $\domain _{min}$ must vanish 
on $\partial \gbar $.  Thus $g_2$ agrees with $f$ on $\partial \gbar$.  Since
\[ \| g \| _1^2 = \| g_1 \| _1^2 + \| g_2 \| _1^2,\]
$g_2$ the desired minimizer. 

Similarly, if $f \in \Hilbert _1$ vanishes on $\partial \gbar $ and $f = f_1\oplus f_2$, with $f_1$ in the $\Hilbert _1$ closure of $\domain _{min}$ and $f_2 \in H_{fin}$,
the harmonic part $f_2$ vanishes on $\partial \gbar$, forcing $f_2 = 0$.

\end{proof}

The Dirichlet problem for $\gbar$ can now be solved.  This approach emphasizes $H_{fin}$, an aspect not discussed in the proof in \cite{Carlson08}.
Use $1_E$ to denote the characteristic function of a set $E$; in this case $E \subset \partial \gbar$.
The existence of a rich collection of partitions ${\cal E}$ is a consequence of \propref{manycor}.

\begin{thm} \label{Eharm}
Suppose ${\cal E} = \{ E(n), n=1, \dots ,N \}$ is a finite partition  of $\partial \gbar $ by clopen sets.
For any function $F  = \sum_{n=1}^N c_n 1_{E(n)}$,
which is a linear combination of the characteristic functions of the sets $E(n)$,
there is a unique $f \in H_{fin}$ with $f = F$ on $\partial \gbar$. 
\end{thm}

\begin{proof}

By \lemref{Afunk} there is a function $g \in \alg$ which agrees with $F$ on $\partial \gbar$.
Since $g \in \Hilbert _1$, \corref{minzer} implies there is an $f \in H_{fin}$ which agrees with $g$ on $\partial \gbar$.
The uniqueness of $f$ follows immediately from the maximum principle.  

\end{proof}

\begin{cor} \label{Dprob}
Suppose $G:\partial \gbar \to \real $ is continuous.  Then there is a unique harmonic function $g:\gbar \to \real$ with $g= G$ on $\partial \gbar$.
\end{cor}

\begin{proof}
Since $\partial \gbar$ is compact, continuous functions $f: \partial \gbar \to \real$ are uniformly continuous.  Consequently,
the functions $F = \sum_{n=1}^N c_n 1_{E(n)}$ from the proof of \thmref{Eharm} are uniformly dense in the continuous functions on $\gbar$. 

Using \thmref{Eharm}, pick a sequence $f_n: \gbar \to \real $ of harmonic functions converging uniformly to $G$ on $\partial \gbar $.
By the maximum principle $ f_n: \gbar \to \real $ is a uniformly Cauchy sequence, which converges uniformly to the desired harmonic function $g$.  
\end{proof}

Note that not every harmonic function $f:\gbar \to \real $ is in $\Hilbert _1$.

\begin{prop}
Suppose $\partial \gbar $ is not a set of isolated points.  Then
there are continuous functions $F:\partial \gbar \to \real $ whose harmonic extensions are not in $\Hilbert _1$.
\end{prop}

\begin{proof}

A function $f \in H_{fin}$ will satisfy a Lipschitz condition.  Suppose $x,y \in \gbar $, and $\gamma $ is a path with length
at most $2d(x,y)$ from $x$ to $y$.  Then
\begin{equation} \label{Lip}
|f(y) - f(x)|^2 = |\int_{\gamma} f'(t) \ dt |^2 \le \int_{\gamma} (f')^2 \int_{\gamma} 1 \le 2d(x,y) \| f \| _1^2 .
\end{equation}    

Suppose $x_0 $ is a limit point of $\partial \gbar $.  Consider the function 
$ F(x) = \sqrt{d(x,x_0)}$.  Since $x^{1/2}$ is continuous on $[0,\infty )$ but $(x^{1/2})' = x^{-1/2}/2$, the function $F$
cannot satisfy the Lipschitz condition \eqref{Lip}.
\end{proof}

\subsection{Level sets}

Assume $\graph _0$ is a connected finite graph, with $E $ and $E^c = \partial \graph _0 \setminus E$  nonempty subsets of $\partial \graph _0$ .
Suppose $f:\graph _0 \to \real$ is harmonic, with $f(v) = 1$ for $v \in E$ and $f(v) > 1$ for $v \in E^c$. 
By the maximum principle, $\partial _{\nu} f(v) > 0$ for $v \in E$.  It will be helpful to have a similar result for  infinite
graphs $\graph $.  

A point $x \in \graph $ is a critical point for $f$ if $x$ is a vertex or $f'(x) = 0$.  
A number $y \in \real $ is a critical value for $f$ if $f^{-1}(y)$ contains a critical point.
Points in the range of $f$ that are not critical values are regular values.
If $f$ is harmonic and $f'(x) = 0$ for some $x $ in an edge $e$, then $f$ is constant on $e$.
Since $\vertexset$ and $\edgeset$ are countable, a harmonic function $f$ has a countable set of critical values.

\begin{lem} \label{finset}
Suppose $f:\gbar \to \real$ is harmonic, and $c$ is a regular value of $f$.
Assume there is no $x \in \partial \gbar $ with $f(x) = c$.  
Then $f^{-1}(c)$ is a finite set.
\end{lem}

\begin{proof}
If $f^{-1}(c)$ were an infinite set, then by compactness there would be an infinite
sequence of distinct points $\{ x_n \} \subset f^{-1}(c)$ converging to a point $z$, with $f(z) = c$.
Since $z \notin \partial \gbar$ and $c$ is a regular value, $z$ is an interior point of some edge $e$.  Since $f$ is harmonic and $x_n \to z$, 
$f$ must be constant on $e$, contradicting the assumption that $c$ is a regular value for $h$.
\end{proof}

\begin{lem} \label{goodt}
Suppose that $E $ and $E^c = \partial \gbar \setminus E $ are nonempty clopen subsets of $\partial \gbar$.
Assume that $f:\gbar \to \real$ is harmonic, with $ f(x) =  C \ge 0$ for $ \ x \in E$, and  $f(x) > C$ for  $ x \in E^c$.
Given $\epsilon > 0$ there is a $t > C$ such that $d(x,E ) < \epsilon $ if $f(x) \le t$.
\end{lem}

\begin{proof}
Since $E^c$ is compact, $y = \min _{x \in {E^c}} f(x) > C$.
Arguing by contradiction, assume there is a sequence $t_n \to C$ and points $x_n \in f^{-1}(t_n)$ with $d(x_n, E ) \ge \epsilon $. 
By compactness of $\gbar$ the sequence $\{x_n \}$ has a subsequential limit $z$, with $f(z) = C$.
Since $z \notin E^c$ and $d(z,E) \ge \epsilon $, $z$ must be in $\graph $, contradicting the maximum principle.
\end{proof}

Using the hypotheses of \lemref{goodt}, for $\epsilon > 0$ select a regular value $t$, with $C < t < \min_{x \in E^c} f(x)$
and $d(x,E ) < \epsilon $ if $f(x) \le t$.
By \lemref{finset} the set $f^{-1}(t)$ is a finite set of points interior to edges of $\graph $.  
Add the vertices $f^{-1}(t)$ to $\graph $, subdivide the corresponding edges, and call the resulting graph $\widetilde{\graph}$.  
The graph $\widetilde{\graph}$ is now replaced by the set
$\graph _t =  \widetilde{\graph} \cap f^{-1}[t,\infty ) $.      

\begin{lem} \label{posgraph}
$\graph _t$ is a subgraph of $\widetilde{\graph}$ with $f^{-1}(t) \subset \partial \graph _t$. Moreover, $\partial _{\nu}f(x) > 0$ when $f(x) = t$.
\end{lem}

\begin{proof}
Suppose $e$ is an edge of $\widetilde{\graph}$,  with vertices $v_1,v_2$.
If $x \in e$ with $f(x) = t$, then $f'(x) \not= 0$ and $\graph _t$ contains only one of $[v_1,x]$ or $[x,v_2]$.
The point $x$ is then a degree one vertex in $\graph _t$, so is a boundary vertex.
If $e$ contains no point $x$ with $f(x) = t$, then $e$ is either entirely in $\graph _t$, or in the complement.  Thus $\graph _t$ is the union of closed
edges of $\widetilde{\graph}$.

Finally, $\partial _{\nu}f(x) > 0$ since $t$ is a minimum value for $f$ on $\graph _t$.  

\end{proof}

The volume of $\graph $ is its Lebesgue measure, i.e. the sum of the edge lengths.  
The volume of the $\epsilon $-neighborhood of the clopen set $E \subset \partial \gbar$ is the  Lebesgue measure of $N_{\epsilon}(E)$.

\begin{prop}
Suppose $\epsilon > 0$, $N_{\epsilon}(E)$ has finite volume, and $t$ is chosen so that $\{ f \le t \} \subset  N_{\epsilon}(E)$.
Then every $x \in \graph $ with $f(x) = t$ can be connected to $\partial \gbar $ by a path $\gamma$ in the set $\{ f \le t \}$,
and $\gamma $ can be chosen to have a single limit point in $\partial \gbar $. 
\end{prop}

\begin{proof}

If $f(x) = t$,  then $x$ is interior to an edge $e(0)$, and $f_{e(0)}'(x) \not= 0$.  Walk in the direction of decreasing $f$ until you hit a vertex $v$.
If $v$ is not a boundary vertex, then since $\partial _{\nu} f_{e(0)}(v) > 0$ and \eqref{Kerck} holds, there is another edge $e(1)$ incident on $v$ with $\partial _{\nu} f_{e(1)} < 0 $.
Walk along $e(1)$, and then continue in this fashion.  Either a boundary vertex is encountered after finitely many steps, or 
there is a path $\gamma $ with infinitely many distinct edges $e(n)$.  By \lemref{approx} the distance from $\gamma (s)$ to $\partial \gbar$ has
limit zero as $n \to \infty $.

Let $z$ be a limit point of $\gamma $.  By \propref{manycor} there is a sequence $\{ {\cal E }_j, j = 1,2,3,\dots \} $ of finite partitions of 
$\partial \gbar $ by clopen sets, with ${\cal E }_{j+1}$ a refinement of ${\cal E }_j$,
such that the diameter of each set in $ {\cal E }_j $ is less than $1/j$.  Let $z \in E_k $ with $E_k \in  {\cal E }_k$.

Since ${\cal E }_k$ is a finite partition by clopen sets there is a $\delta > 0$ such that $d(x,y) > \delta $ if $x,y$ lie in 
distinct sets of ${\cal E }_k$.  Since the sum of the lengths of the edges in $\gamma $ is finite, any limit of $\gamma $
must lie in $E_k$, for $k = 1,2,3, \dots $.  Since the diameters of the sets $E_k$ have limit zero, any limit point of $\gamma $ must be $z$.

\end{proof}

Suppose $t_1$ and $t_2$ are as above, with $t_2 < t_1$.  Consider the set $\graph _2 = [f^{-1}(t_2), f^{-1}(t_1)]$.

\begin{lem} \label{ineqout}
$\graph _2$ is a finite graph containing no boundary vertices of $\graph$.
If $f^{-1}(t_1)$ and $f^{-1}(t_2)$ are considered as boundary vertices of $\graph _{t_1}$ and $\graph _{t_2}$ respectively, then
\[\sum _{v \in f^{-1}(t_1)}\partial _{\nu} f = \sum _{v \in f^{-1}(t_2)} \partial _{\nu} f .\]
\end{lem}

\begin{proof}

$\graph _2$ consists of a collection of closed edges as in the proof of \thmref{posgraph}.
The values $t_1$ and $t_2$ are chosen to be positive, but smaller than $f(x)$ for $x \in E^c$.
Thus $\graph _2$ contains no boundary vertices of $\graph $.

If $\graph _2$ had infinitely many distinct edges, there would be a sequence $x_n$ from distinct edges 
with $f(t_2) \le f(x_n) \le f(t_1)$, but with $z = \lim_{n \to \infty} x_n \in E$ by \lemref{approx}.  
This would force $f(z) = C$, which is impossible.

Since $\graph _2$ contains no boundary vertices of $\graph $, \lemref{IBP} gives
\[0 = \int_{\graph _2} - f''\cdot 1 = \sum _{v \in f^{-1}(t_2)} \partial _{\nu} f (v)+  \sum _{v \in f^{-1}(t_1)} \partial _{\nu} f(v) .\]
But the outward pointing derivatives  $\partial _{\nu} f(v)$ for $v \in f^{-1}(t_1)$ flip sign when considered as boundary vertices
for $\graph _{t_1}$, finishing the proof.
    
\end{proof}

\section{Differential operators on $L^2(\graph)$ }

Recall that $\domain _{max}$ denotes the continuous real valued functions $f$ on $\graph $ 
which have absolutely continuous derivatives on each edge $e$, with $f$ and $f''$ in $L^2(\graph )$, and which   
satisfy \eqref{Kerck} at each interior vertex.  $\domain _{min}$ denotes the functions $f \in \domain _{max}$ 
which are infinitely differentiable on the edges of $\graph$, with compact support $supp (f)$ in the interior of $\graph$.
$S_{min}$ will be the operator $\Laplace$ acting on $L^2(\graph )$ with domain $\domain _{min}$.
Any symmetric extension of $S_{min}$ will have an adjoint which is a restriction of $S_{min}^*$.

\begin{prop}
$S_{min}$ is a nonnegative symmetric operator on $L^2(\graph )$.  The adjoint $S_{min}^*$ is the operator $\Laplace$ with domain $\domain _{max}$.
\end{prop}

\begin{proof}
If $f,g \in \domain _{min}$ there is a finite subgraph $\graph _0$ containing $supp(f) \cup supp(g) $ such that
\[\int_{\graph} (-D^2f)g = \int_{\graph _0} (-D^2f)g = \int_{\graph _0} f'g' =   \int_{\graph} f(-D^2g).\]
Thus $S_{min}$ is a nonnegative and symmetric. For each edge $e \in \edgeset$, the domain $\domain _{min}$ includes the $C^{\infty}$ functions with 
compact support in the interior of $e$. By \cite[p. 1294]{Dunford} or \cite[p. 169-171]{Kato} $S_{min}^*$ acts by $\Laplace $, and 
functions in the domain of $S_{min}^*$ have absolutely continuous derivatives on each edge $e$, with $S_{min}^*f = \Laplace f \in L^2(\graph )$.

Suppose now that $v$ is an interior vertex with incident edges $e_n = [v,w_n]$ for $n = 1,\dots ,e_N$.  Assume that
$f \in \domain _{min}$ has support in $U = \bigcup _{n=1}^N e_n$,.
Choose such an $f$ with $f=1$ in a neighborhood of $v$
and with $f(w_n)$ and $f'(w_n) $ vanishing in a neighborhood of $w_n$.  For $g $ in the domain of $S_0^*$,
\[0 = \int_{U} (-f'')g - \int _{U}f(-g'') = \sum_{e \sim v} [ f_e(v) \partial _{\nu}g_e(v) - g_e(v)\partial _{\nu} f_e(v)] = \sum_{e \sim v} \partial _{\nu}g_e(v).\]

Next, for distinct $m,n$, choose $f$ with $f(v) = 0$, but with $\partial _{\nu} f_{e_m} = -1$, $\partial _{\nu} f_{e_n} = 1$, and  $\partial _{\nu} f_{e_j} = 0$ for $j \not= m,n$.
Then $g $ in the domain of $S_0^*$ is continuous at $v$, since
\[ 0 =   \sum_{e \sim v} [ f_e(v) \partial _{\nu}g_e(v) - g_e(v)\partial _{\nu} f_e(v)] = g_{e_m}(v) - g_{e_n}(v) .\]

\end{proof}

The domain $\domain _{min}$ will now be extended, with the aim of finding novel self-adjoint operators $\Laplace$.
Start with a finite partition ${\cal E}(0) = \{ E_{n,0}, n = 0, \dots , N \}$ of $\partial \gbar $ by nonempty clopen sets.  
Assume that for some $\delta > 0$ the neighborhoods $N_{\delta }(E_{n,0})$ are pairwise disjoint, and 
have finite volume for $n = 1,\dots , N$; neighborhoods of $E_{0,0}$ may have infinite volume.
Also choose a collection ${\cal H} = \{h_n, n = 1, \dots , N\} $ of functions, with $h_n$ defined and 
harmonic in $N_{\delta }(E_{n,0})$; initially it will suffice to assume that $h_n \in L^2( N_{\delta }(E_n))$. 

Define a domain $\domain _{{\cal E}(0),{\cal H}}$ of functions which are constant multiples of $h_n$ in some neighborhood $N_{\epsilon }(E_{n,0})$. 
More precisely $\phi \in \domain _{{\cal E}(0),{\cal H}}$ if $\phi \in \domain _{max} $ and for some $\epsilon$ with $0 < \epsilon < \delta $, 

(i) $\phi (x) = 0$ for $x \in N_{\epsilon }(E_{0,0})$, and

(ii) there are constants $c_{n}$ (depending on $\phi $) such that $\phi (x) = c_n h_n(x)$ for $x \in N_{\epsilon}(E_{n,0})$.

Note that $\domain _{{\cal E}(0),{\cal H}}$ is a dense (algebraic) subspace of $L^2(\graph)$.  
The finite volume assumption for $N_{\delta}(E_n)$, $n =1,\dots ,N,$ insures that 
$\phi \in L^2(\graph)$ if the $h_n \in \cal H$ are bounded.   
Let $S_{{\cal E}(0),{\cal H}}$ denote the operator on $L^2(\graph )$ acting by $\Laplace$ 
with domain $\domain _{{\cal E}(0),{\cal H}}$.  The collection ${\cal H}$ serves as boundary conditions on ${\cal E}(0)$. 
(If $vol (\graph ) $ is finite, $E_0$ and condition (i) may be dropped.)

The partition and domain for $\Laplace$ will be extended inductively.
Suppose ${\cal E}(j) = \{ E_{m,j} \}$ is an already defined finite partition of $\partial \gbar$ by nonempty clopen sets.
Form the partition ${\cal E}(j+1)$ by partitioning the sets $E_{m,j}$ into two nonempty disjoint clopen sets $E_{m,j,1}$ and $E_{m,j,2}$
and putting these into ${\cal E}(j+1)$.  If $E_{m,j}$ may not be partitioned in this way, add it unchanged to ${\cal E}(j+1)$. 

The associated collection ${\cal H}(j+1)$ of harmonic functions is obtained by assigning the function $h_{m,j}$ previously assigned to $E_{m,j}$
to both sets  $E_{m,j,1}$ and $E_{m,j,2}$.  As above, $\phi \in \domain _{{\cal E}(j+1),{\cal H}(j+1)}$ if there are constants $c_{m,j,1}$ and $c_{m,j,2}$ such that 
$\phi (x) = c_{m,j,1} h_{m,j}(x)$ for $x$ in some $\epsilon $ neighborhood $N_{\epsilon}(E_{n,1})$ and $\phi (x) = c_{m,j,2} h_{m,j}(x)$ for $x \in N_{\epsilon}(E_{n,2})$.
Note that the collection ${\cal H}(j+1)$ of harmonic functions is only assigning a function $h_{m,j}$ to the subsets of $E_{m,j}$, 
although as sets are split the constants need not be the same.  
However, if $c_{m,j,1} = c_{m,j,2}$ for all $m$, a subspace of $\domain _{{\cal E}(j+1),{\cal H}(j+1)}$ is naturally identified with 
$\domain _{{\cal E}(j),{\cal H}(j)}$.  The operators $S_{{\cal E}(j),{\cal H}(j)}$ act by $\Laplace$ with the increasing domains $\domain _{{\cal E}(j),{\cal H}(j)}$.
The operator $S_U$ will act by $\Laplace$ with the domain $\bigcup _j \domain _{{\cal E}(j),{\cal H}(j)}$.

\begin{thm} \label{Symops}
The operators $S_{{\cal E}(j),{\cal H}(j)}$ and $S_U$ are symmetric.  They have self-adjoint extensions.  
\end{thm}

\begin{proof}
Suppose $f,g \in \domain _{{\cal E}(j),{\cal H}(j)}$.  Since both are harmonic in a neighborhood of $\partial \gbar $, 
there is a finite graph $\graph _{\epsilon}$ as in \lemref{approx} such that \lemref{IBP} gives
\[\int_{\graph} (D^2f)g - \int_{\graph} f(D^2g)  = \int_{\graph _{\epsilon}} (D^2f)g - \int_{\graph _{\epsilon}} f(D^2g)\]
\[= \sum _{v \in \partial \graph _{\epsilon } }\sum_{e \sim v}[f(v) \partial _{\nu} g_e(v) - g(v) \partial _{\nu}f_e(v)] .\]
$\graph _{\epsilon}$ can be chosen so that near the boundary vertices $v \in N_{\epsilon} (E_{m,j})$ there are constants $c_f,c_g$ such that $f = c_fh_n$ and $g = c_gh_n$.
Symmetry follows from 
\[  f(v)\partial _{\nu} g_e(v) - g(v) \partial _{\nu}f_e(v) 
= c_fc_g h_n(v) \partial _{\nu} h_{n,e}(v) - c_gc_f h_n(v) \partial _{\nu} h_{n,e}(v) = 0.\]

The operators $S_{{\cal E}(j),{\cal H}(j)}$ and $S_U$ are densely defined and symmetric on the real Hilbert space $L^2(\graph )$.  
All such operators have  \cite[p. 349]{RS2} self-adjoint extensions on $L^2_{\real}(\graph )$ or \cite[p. 1231]{Dunford} on the 
complexification $L_{\complex}^2(\graph)$.

\end{proof}

It is often useful to know that a densely defined symmetric operator $S$ is nonnegative, since $S$ will then have a
distinguished self-adjoint extension, the Friedrichs extension. 
To establish nonnegativity of the symmetric operators $S_{{\cal E}(j),{\cal H}(j)}$, restrictions are placed on the harmonic functions.
Choose a collection ${\cal K} = \{ k_n, n  = 1,\dots ,N \}$ of harmonic functions as before, with each $k_n$ satisfying
(i) each $k_n$ has a constant value $C_n >0$ on $E_n$, (ii) $k_n \in H_{fin}$, and (iii) $\partial _{\nu} k_n(x) > 0$ 
for $x$ in the level sets $\{ k_n(x)  = t_n  \}$ for $t_n$ sufficiently close to $C_n$.  The existence of such functions $k_n$ is established
in \lemref{posgraph} and \thmref{Eharm}.
 
\begin{thm} \label{Rops}
$S_{{\cal E}(j),{\cal K}(j)}$ and the corresponding $S_U$ are nonnegative.  There are functions $f$ in the domain of  $S_{{\cal E}(j),{\cal K}(j)}$ with quadratic form 
$\langle S_{{\cal E}(j),{\cal K}(j) } f,f \rangle _2$ strictly larger than $\int_{\graph } (f')^2$.
\end{thm}

\begin{proof}
Replace the graph $\graph _{\epsilon }$ of \thmref{Symops} with a graph $\graph _T$ bounded by the level sets $\{ k_n(x)  = t_n  \}$
and the set $\{ x, d(x,E_0) = t_0 \}$.  By \lemref{IBP}, for $t_m$ sufficiently small, $m=0,\dots ,N$ the quadratic form for  $S_{{\cal E}(j),{\cal K}(j)}$ is 
\begin{equation} \label{Qform}
\int_{\graph } (-D^2f) f =  \sum_{n=1}^N \sum _{v \in k_n^{-1}(t_n)} f\partial _{\nu} f + \int _{\graph _T} (f')^2
\end{equation}
\[ =   \sum_{n=1}^N  \sum _{v \in k_n^{-1}(t_n)} c_{n,f}^2 k_n(v)  \partial _{\nu} k_n(v) + \int _{\graph _T} (f')^2 \ge 0.\]

By \lemref{ineqout}, the sums $ \sum _{v \in k_n^{-1}(t_n)} \partial _{\nu}k_n(v)$ will have a fixed positive value as $t_n \downarrow C_n$.
Consequently, the quadratic form for $S_{{\cal E}(j),{\cal K}(j)}$ will be strictly larger than $\int _{\graph}(f')^2$ for functions in $\domain _{{\cal E}(j),{\cal K}(j)}$ with 
$c_{n,f}^2 > 0$.

\end{proof}

When $S_{{\cal E}(j),{\cal K}(j)}$ is nonnegative it has \cite[pp. 313-326]{Kato} a nonnegative self-adjoint Friedrichs extension $\dop _{{\cal E}(j),{\cal K}(j)}$.
The quadratic form for $S_{{\cal E}(j),{\cal K}(j)} + I$ provides an inner product,
$\langle f,g \rangle _F = \langle S_{{\cal E}(j),{\cal K}(j)} f,g \rangle _2 + \langle f,g \rangle _2 $.
Then  \cite[p. 322]{Kato} the completion of $ \domain _{{\cal E}(j),{\cal K}(j)}$ with respect to the inner product  $\langle f,g \rangle _F$ yields a 
closed form $\tau [f,g]$ extending $\langle S_{{\cal E}(j),{\cal K}(j)}f,g \rangle _2$, with the associated Friedrichs extension $\dop _{{\cal E}(j),{\cal K}(j)}$
having a domain contained in the form domain of $\tau $.

\begin{cor}
If $f$ is in the completion of $ \domain _{{\cal E}(j),{\cal K}(j)}$ with respect to the inner product  $\langle f,g \rangle _F$,
then $f$ is continuous on $\gbar$ and absolutely continuous on the edges of $\graph $, with $f' \in L^2(\graph )$.

Suppose for $n = 1,\dots ,N$ the functions $h_n \in {\cal H}$ are constants.  Then $\tau [f,f] =  \int_{\graph} (f')^2 $.
\end{cor} 

\begin{proof}
The function $f$ is the limit of a Cauchy sequence $f_n \in \domain _{{\cal E}(j),{\cal K}(j)}$ with respect to the norm of the inner product $\langle f,g \rangle _F$.
By \eqref{Qform}, 
\[\| f \|_F ^2 \ge  \int_{\graph} f ^2+ (f')^2 .\]
Following the proof of \propref{isHilbert} with $\int_{\partial \gbar}  f^2 \ d\mu $ replaced by $\int_{\graph} f^2$, one concludes
that $f$ is continuous on $\gbar$ and absolutely continuous on the edges of $\graph $, with $f' \in L^2(\graph )$. 
In case $h_n$ is replaced by a constant, \eqref{Qform} simplifies to 
\[\langle f,f \rangle _F = \int_{\graph} f^2 + (f')^2.\]

\end{proof}

The quadratic forms thus provide a way to distinguish the Friedrichs extensions for various harmonic functions $h$.

\begin{cor} The Friedrichs extensions of $\dop _{{\cal E}(j),{\cal K}(j)}$ are distinct from the Friedrichs extensions of $S_{{\cal E},{\cal H}}$ 
obtained when the functions $h_n \in {\cal H}$ are constants. 
\end{cor}

\bibliographystyle{amsalpha}

\begin{thebibliography}{10}

\bibitem{Avdonin}
S.~Avdonin.
\textit{Control, observation and identification problems for the wave equation on metric graphs}
IFAC-PapersOnline 52 no. 2 (2019) 52-57

\bibitem{BK}
G. ~Berkolaiko and P.~Kuchment. 
\textit{Introduction to Quantum Graphs.} 
American Mathematical Society, Providence, 2013.

\bibitem{Carlson08}
R. ~Carlson.
\textit{Boundary value problems for infinite metric graphs.}
Analysis on graphs and its applications,
Proc. Sympos. Pure Math. 77, 355--368, 2008. 

\bibitem{Cartwright}
D. ~Cartwright, P.~Soardi, and W.~Woess.
\textit{Martin and end compactifications for nonlocally finite graphs.}
Transactions of the American Mathematical Society 338, num. 2, 679--693, 1993.

\bibitem{DagerZ}
R.~Dager and E.~Zuazua  
\textit{Wave Propagation, Observation and Control in $1$-d Flexible Multi-structures.} 
Springer, 2006.


\bibitem{Doyle}
P.~Doyle and J.~Snell  
\textit{Random Walks and Electric Networks.} Mathematical Association of America, Washington, 1984.

\bibitem{Dunford}
N.~Dunford and J.~Schwartz.
\textit{Linear Operators, Part II}.
\newblock Interscience, New York, 1964.


\bibitem{Exner18}
P.~Exner, A.~Kostenko, M.~Malamud and H.~Neidhardt.
\textit{Spectral theory of infinite quantum graphs.}
\newblock {Ann. Henri Poincare 19:3457-3510, 2018.} 

\bibitem{Georg11}
A.~Georgakopoulos.
\textit{Graph topologies induced by edge lengths.}
\newblock {Discrete Mathematics 311:1523--1542, 2011.} 

\bibitem{GHKLW}
A.~Georgakopoulos, S.~Haesler, M.~Keller, D.~Lenz, R.~Wojciechowski.
\textit{Graphs of finite measure.}
\newblock {Journal de Mathematiques Pures et Appliquees 103, 2015.} 

\bibitem{Graff}
K.~Graff. 
\textit{Wave Motion in Elastic Solids.} 
Dover, 1991.

\bibitem{Hocking}
J.~Hocking and G.~Young. 
\textit{Topology.} Dover, 1988.

\bibitem{JP}
P.~Jorgensen and E.~Pearse.
\textit{Operator theory of electrical resistance networks.}
arXiv: 0806.3881v1 2008.

\bibitem{Kato}
T. ~Kato.
\textit{Perturbation Theory for Linear Operators}.
Springer-Verlag, New York, 1995.

\bibitem{KolmF}
A.~Kolmogorov and S.~Fomin. 
\textit{Introductory Real Analysis.} Dover Publications, New York, 1975.

\bibitem{KMN}
A.~Kostenko, D.~Mugnolo, N.~Nicolussi.
\textit{Self-adjoint and Markovian extensions of infinite quantum graphs.}
\newblock {arXiv 1911.04735v1, 2019.} 

\bibitem{Lagnese}
J. ~Lagnese, G.~Leugering and E.~Schmidt.
\textit{Modeling, Analysis and Control of Dynamic Elastic Multi-Link Structures.} 
Birkh\"auser, 1994.

\bibitem{RS2}
M.~Reed and B.~Simon. 
\textit{Methods of Modern Mathematical Physics II.} 
Academic Press, New York, 1975.

\bibitem{Royden}
H.~Royden. 
\textit{Real Analysis.} 
Macmillan, New York, 1988.

\end{thebibliography}

\end{document}